\documentclass{article}
\usepackage[utf8]{inputenc}
\usepackage[a4paper,top=3cm,bottom=3cm,left=2.5cm,right=2.5cm]{geometry}

\usepackage{amsthm}
\usepackage{amssymb}
\usepackage{enumerate}
\usepackage{tikz-cd}
\usepackage{mathtools}
\tikzcdset{arrow style=math font}
\usepackage{braket}
\usepackage{xcolor}
\usepackage{comment}
\usepackage[all,cmtip]{xy}
\usepackage{hyperref}
\usepackage{cleveref}
\usepackage{todonotes}

\def\C{\ensuremath{\mathbb{C}}}

\def\P{\ensuremath{\mathbb{P}}}
\def\R{\ensuremath{\mathbb{R}}}
\def\Z{\ensuremath{\mathbb{Z}}}
\def\K{\ensuremath{\mathbb{K}}}

\def\AA{\ensuremath{\mathcal A}}

\def\CC{\ensuremath{\mathcal C}}
\def\DD{\ensuremath{\mathcal D}}
\def\EE{\ensuremath{\mathcal E}}
\def\FF{\ensuremath{\mathcal F}}

\def\KK{\ensuremath{\mathcal K}}

\def\MM{\ensuremath{\mathcal M}}

\def\OO{\ensuremath{\mathcal O}}
\def\PP{\ensuremath{\mathcal P}}

\def\TT{\ensuremath{\mathcal T}}
\def\UU{\ensuremath{\mathcal U}}

\def\vv{\ensuremath{\mathbf v}}

\def\rH{\mathrm H}
\def\rVb{\mathrm V^{\mathrm{b}}}
\def\rBb{\mathrm B^{\mathrm{b}}}
\def\rKb{\mathrm K^{\mathrm{b}}}
\def\rVbdg{\mathrm V^{\mathrm{b}}_{\mathrm{dg}}}
\def\rBbdg{\mathrm B^{\mathrm{b}}_{\mathrm{dg}}}

\def\Db{\mathop{\mathrm{D}^{\mathrm{b}}}\nolimits}

\def\Ext{\mathop{\mathrm{Ext}}\nolimits}

\def\Ku{\mathop{\mathcal{K}u}\nolimits}
\def\Hom{\mathop{\mathrm{Hom}}\nolimits}
\def\Ext{\mathop{\mathrm{Ext}}\nolimits}
\def\Ddg{\mathop{\mathrm{D}^{\mathrm{b}}_{\mathrm{dg}}}\nolimits}
\def\Cdg{\mathop{\mathrm{C}^{\mathrm{b}}_{\mathrm{dg}}}\nolimits}

\theoremstyle{plain}
  \newtheorem{thm}{Theorem}[section]
  \newtheorem{cor}[thm]{Corollary}
  \newtheorem{lem}[thm]{Lemma}
  \newtheorem{prop}[thm]{Proposition}

  \theoremstyle{definition}
\newtheorem{dfn}[thm]{Definition}
\newtheorem{rmk}[thm]{Remark}

\newtheorem{ex}[thm]{Example}

\newtheorem{innercustomthm}{Assumption}
\newenvironment{customthm}[1]
  {\renewcommand\theinnercustomthm{#1}\innercustomthm}
  {\endinnercustomthm}
  
\newtheorem*{str}{Strategy of the proofs}
\newtheorem*{rw}{Motivations and related works}
\newtheorem*{plan}{Plan of the paper}
\newtheorem*{ack}{Acknowledgements}  

\begin{document}

\title{Derived categories of hearts on Kuznetsov components}
\author{Chunyi Li, Laura Pertusi and Xiaolei Zhao}
\date{\today}

\maketitle 

\begin{abstract}
We prove a general criterion which guarantees that an admissible subcategory $\KK$ of the derived category of an abelian category is equivalent to the bounded derived category of the heart of a bounded t-structure. As a consequence, we show that $\KK$ has a strongly unique dg enhancement, applying the recent results of Canonaco, Neeman and Stellari. We apply this criterion to the Kuznetsov component $\Ku(X)$ when $X$ is a cubic fourfold, a Gushel--Mukai variety or a quartic double solid. In particular, we obtain that these Kuznetsov components have strongly unique dg enhancement and that exact equivalences of the form $\Ku(X) \xrightarrow{\sim} \Ku(X')$ are of Fourier--Mukai type when $X$, $X'$ belong to these classes of varieties, as predicted by a conjecture of Kuznetsov.  
\end{abstract}

\section{Introduction}

The triangulated category is an important tool from the point of view of algebraic geometry. In fact, the bounded derived category of coherent sheaves on a smooth projective variety $X$ has a triangulated structure and encodes much information about the geometry of $X$. In 1997, Bondal and Orlov proved that smooth projective varieties with ample (anti)canonical
bundle and equivalent bounded derived categories are isomorphic \cite{BO}. Similar reconstruction statements, called Categorical Torelli theorems, have been obtained for admissible subcategories of the bounded derived category, arising as residual components of exceptional collections in semiorthogonal decompositions, of certain Fano threefolds and fourfolds \cite{BMMS, soheylaetall, PY, BT, APR, HR, BLMS, LPZ_twisted} (see \cite{PS_survey} for a survey on this topic).

It is often convenient to associate higher categorical structures to a triangulated category $\TT$. The easiest one yields the notion of dg enhancement, which is a dg category with the same set of objects as $\TT$ and whose homotopy category is equivalent to $\TT$. One first advantage of passing to the dg level is that we gain a functorial notion of cone of a morphism \cite[Paragraph 2.9]{Dri}.

Not all triangulated categories have a dg enhacement (see for instance \cite{RV:nodg} for counterexamples). However, if $\AA$ is an abelian category, then an enhacement of $\Db(\AA)$ is given by the Verdier quotient of the dg category of bounded complexes in $\AA$ over its full dg subcategory of acyclic complexes. 

Once an enhancement exists, it is also natural to ask whether it is unique. This has been proved for $\Db(\AA)$ by Lunts and Orlov in 2009 when $\AA$ is a Grothendieck abelian category with a small set of compact generators \cite{LO}, generalized in 2015 by Canonaco and Stellari for any Grothendieck abelian category $\AA$ in \cite{CS_unique}, and finally proved in 2018 for any abelian category by Antieau in \cite{Antieau}. Recently Canonaco, Neeman and Stellari have given a new proof of this result in \cite{CNS}.

The first result of this paper is a criterion which guarantees that an admissible subcategory $\KK$ of the derived category of an abelian category is itself equivalent to the derived category of an abelian category. Clearly \cite{Antieau} implies that $\KK$ has a unique enhancement. Using the construction in \cite{CNS} we can further show that $\KK$ as in Theorem \ref{thm_derivedcategoryofheart_intro} has a unique enhancement in a strong sense (see Definition \ref{def_uniqueenh}), as stated below. 
\begin{thm}[Theorem \ref{thm_derivedcategoryofheart}, Theorem \ref{thm_stroglyunique}] \label{thm_derivedcategoryofheart_intro}
Let $\TT$ be the derived category of an abelian category. Let $\KK$ be an admissible subcategory of $\TT$ having a stability condition $\sigma=(\AA, Z)$, whose heart $\AA$ is induced from a heart on $\TT$ and satisfying Assumption \ref{star}. Then there is an exact equivalence $\Db(\AA) \xrightarrow{\sim} \KK$. Moreover, we have that $\KK$ has a strongly unique enhancement.
\end{thm}

In the second part of this paper we apply Theorem \ref{thm_derivedcategoryofheart_intro} to several interesting geometric examples. The first and most famous is represented by the Kuznetsov component of a cubic fourfold $X$, defined as the full admissible subcategory
$$\Ku(X):= \langle \OO_X, \OO_X(1), \OO_X(2) \rangle^{\perp}=\lbrace E \in \Ku(X): \Hom_{\Db(X)}(\OO_X(i), E)=0 \text{ for every }i=0,1,2 \rbrace$$
of $\Db(X)$, where $\OO_X, \OO_X(1), \OO_X(2)$ are line bundles on $X$ (see Example \ref{ex_cubic4}). We also consider the Kuznetsov components of Gushel--Mukai varieties and of quartic double solids, which are defined analogously (see Examples \ref{ex_GMeven}, \ref{ex_GModd}, \ref{ex_qds}). We have the following main result.

\begin{thm}[Theorem \ref{thm_application}] \label{thm_application_intro} 
Let $\Ku(X)$ be the Kuznetsov component of a cubic fourfold or of a Gushel--Mukai variety or of a quartic double solid. Then there is an equivalence $\Ku(X) \cong \Db(\AA)$, where $\AA$ is the heart of a stability condition on $\Ku(X)$. Moreover, we have that $\Ku(X)$ has a strongly unique enhancement.
\end{thm}

Theorem \ref{thm_application_intro} has interesting consequences on the characterization of exact equivalences between these Kuznetsov components as functors of Fourier--Mukai type (see Definition \ref{def_FMtype}). In fact, the most important exact functors in the geometric context are of Fourier--Mukai type. In 1996, Orlov proved that every exact  fully faithful functor with adjoints between the bounded
derived categories of coherent sheaves on smooth projective varieties is of Fourier--Mukai type \cite{O:eqofdcat}.  Since then this result have been further generalized, see \cite{Kawamata:eqofdgstack, CS_twistedFM, LO, CS_supported}. In particular, the key point in \cite{LO} was showing the existence of a dg lift of the functor (see Definition \ref{def_dglift}) to the enhacements, which implies it is of Fourier--Mukai type by the work of To\"en \cite{To}.

In our setting we can prove a version of Orlov's result for the studied Kuznetsov components.

\begin{thm} \label{thm_FMtype}
Let $X_1$, $X_2$ be two cubic fourfolds or Gushel--Mukai varieties of even dimension (resp.\ Gushel--Mukai varieties of odd dimension or quartic double solids). Then every fully faithful exact functor (resp.\ exact equivalence) $\Ku(X_1) \to \Ku(X_2)$ is of Fourier--Mukai type.
\end{thm}

\begin{rw}
In \cite[Definition 3.1]{Kuz_hpd} Kuznetsov defined the notion of splitting functor which is a generalization of that of fully faithful functor. Motivated by Orlov's result, he conjectured in \cite[Conjecture 3.7]{Kuz_hpd} that a splitting functor between bounded derived categories of coherent
sheaves on smooth projective varieties is of Fourier--Mukai type. Note that functors of Fourier--Mukai type are splitting functors. Thus Theorem \ref{thm_FMtype} proves the above mentioned conjecture in the considered geometric cases.  

In the case of the quartic double solid it was observed in \cite[Theorem 7.2]{BT} that Theorem \ref{thm_FMtype} implies the failure of original Fano threefolds Kuznetsov's Conjecture \cite[Conjecture 3.7]{Kuz_Fano3}. Note that Fano threefolds Conjecture has been disproved in \cite{Zhang} and \cite{BP}, independently, in a stronger sense, namely that the Kuznetsov component of a quartic double solid is never equivalent to that of a Gushel--Mukai threefold. 

In \cite{BZ} the authors proved the formality conjecture for semistable objects in the bounded derived category of a K3 surface, using Orlov's result on strongly uniqueness of the enhancement. In the case of cubic fourfolds and Gushel--Mukai varieties of even dimension, the formality conjecture follows from the general results in \cite{Davison}. Nevertheless, Theorem \ref{thm_application_intro} could be useful to provide a direct and simpler proof of this conjecture in these cases. Moreover, the description in Theorem \ref{thm_application_intro} as the bounded derived category of a heart of a stability condition makes the Kuznetsov component much more explicit and manageable.

An interesting question arisen in \cite{CNS} is whether there exist admissible subcategories of the bounded derived category of coherent sheaves on a smooth projective scheme over a field with a non-unique enhancement. Theorem \ref{thm_derivedcategoryofheart_intro} could be helpful to find an answer to this question. 

Finally we believe that Theorem \ref{thm_derivedcategoryofheart_intro} could be applied to the Kuznetsov component of a cubic threefold, although we cannot yet show this, because of the lack of a control of the semistable objects (see Remark \ref{rmk_cubic3}).
\end{rw}

\begin{str}
In \cite{Beilinson1987} Beilinson constructed a functor, known as realization functor, from the derived category of an abelian category in a triangulated category $\TT$ to $\TT$. To prove the first part of Theorem \ref{thm_derivedcategoryofheart_intro} we show that the realization functor is an equivalence under suitable assumptions, which are listed in Assumption \ref{star}. To summarize, we require the existence of a stability condition on the admissible subcategory $\KK$ with a dense set of semistable objects and whose heart has homological dimension $\leq 2$. Assuming this, we show that the $\Hom^2_{\KK}$ between objects in $\AA$ are generated by elements in $\Ext^1_{\AA}$ through Yoneda's composition, which implies the realization functor is an equivalence by \cite[IV, Exercise 2]{GM}.

The second part of Theorem \ref{thm_derivedcategoryofheart_intro} follows from the first part of Theorem \ref{thm_derivedcategoryofheart_intro}, the construction in \cite{CNS} of the quasi-isomorphism between the enhancements and a criterion in \cite{CS} for the extension of isomorphisms of functors.

We prove that Assumption \ref{star} holds for the Kuznesov component of a cubic fourfold and of a Gushel--Mukai variety of even dimension using \cite{BLMS, BLMNPS, PPZ}. In the case of Gushel--Mukai varieties of odd dimension and quartic double solids, stability conditions are known to exist by \cite{BLMS}. We make use of \cite{PR} and \cite{PY} to control the homological dimension of the heart, and of a paper in preparation \cite{PPZ_enriques} where we show the density of the set of semistable objects. This provides the proof of Theorem \ref{thm_application_intro}.

We remark in Proposition \ref{prop_FMtype} that the strongly uniqueness of the enhancement implies equivalences have a dg lift, and these are of Fourier--Mukai type. This implies Theorem \ref{thm_FMtype}.
\end{str}

\begin{plan}
In Section \ref{sec_prelim} we recollect the introductory material on enhancements, Fourier--Mukai functors and stability conditions we need in the next, and the definitions of the Kuznetsov components of cubic fourfolds, Gushel--Mukai varieties and quartic double solids. In Sections  \ref{sec_hearts} and \ref{sec_enhanc} we prove Theorem \ref{thm_derivedcategoryofheart_intro}. In Section \ref{sec_FMfunctors} we explain how to deduce from the second part of Theorem \ref{thm_derivedcategoryofheart_intro} the characterization of equivalences as Fourier--Mukai functors. Section \ref{sec_applications} is devoted to the proof of Theorems \ref{thm_application_intro} and \ref{thm_FMtype}. 
\end{plan} 

\begin{ack}
We wish to thank Paolo Stellari for many useful conversations and for suggesting the strategy in Section \ref{sec_enhanc} and the application to the characterization of equivalences.

C.L.\ is supported by the Royal Society URF$\backslash$R1$\backslash$201129 “Stability condition and application in algebraic geometry”. L.P.\ is supported
by the national research project PRIN 2017 Moduli and Lie Theory. X.Z.\ is partially supported by the Simons Collaborative Grant 636187, NSF grant DMS-2101789, and NSF FRG grant DMS-2052665.
\end{ack}

\section{Preliminaries on dg enhancements, stability conditions, and Kuznetsov components}  \label{sec_prelim}
 
In this section we recollect some definitions and known results on dg enhancements and stability conditions. Finally we list the examples of geometric categories to investigate in this paper.

\subsection{Enhancements and Fourier--Mukai functors} 
Let $\K$ be a field. Recall that a differential graded (dg) category is a $\K$-linear category $\EE$ such that for every pair of objects $A, B \in \EE$ the space of morphisms $\Hom_{\EE}(A, B)$ has the structure of $\Z$-graded $\K$-module with differential $d \colon \Hom_{\EE}(A, B) \to \Hom_{\EE}(A, B)$ of degree $1$ and such that the composition maps $\Hom_{\EE}(B, C) \otimes_{\K} \Hom_{\EE}(A, B) \to \Hom_{\EE}(A, C)$ are morphisms of complexes for every $A, B, C \in \EE$. 

The homotopy category of a dg category $\EE$, denoted by $\rH^0(\EE)$, is the category with the same set of objects as $\EE$ and such that $\Hom_{\rH^0(\EE)}(A, B)=\rH^0(\Hom_{\EE}(A, B))$ for every $A, B \in \EE$. 

Note that if $\EE$ is a pretriangulated dg category (see \cite[Definition 1.7]{CS_tour}), then $\rH^0(\EE)$ is triangulated. In this case, we have the following definitions.

\begin{dfn}
A (dg) enhancement of a triangulated category $\TT$ is a pair $(\EE, \epsilon)$, where $\EE$ is a pretriangulated dg category and $\epsilon \colon \rH^0(\EE) \to \TT$ is an exact equivalence.
\end{dfn}

Recall that a dg functor $F \colon \EE \to \EE'$ between two dg categories $\EE$, $\EE'$ is a functor such that for every pair of objects $A, B \in \EE$, the map $F_{A, B} \colon \Hom_{\EE}(A, B) \to \Hom_{\EE'}(F(A), F(B))$ is a morphism of complexes of $\K$-modules. A dg functor $F$ is a quasi-equivalence if $F_{A, B}$ is a quasi-isomophism for every $A, B \in \EE$ and $\rH^0(F)$ is an equivalence. 

We denote by Hqe the localization of the category of small dg categories with respect to quasi-equivalences. Morphisms in Hqe are called quasi-functors.

\begin{dfn} \label{def_uniqueenh}
A triangulated category $\TT$ has a unique enhancement if given two enhancements $(\EE, \epsilon)$, $(\EE', \epsilon')$ there exists a quasi functor $F \colon \EE \to \EE'$ such that the induced exact functor $\rH^0(F)$ is an equivalence. We say that $\TT$ has a strongly unique enhacement if in addition $F$ can be chosen with the property that there is an isomorphism of functors $\epsilon' \circ \rH^0(F) \cong \epsilon$.  
\end{dfn}

Let $\Phi \colon \Db(X) \to \Db(X')$ be an exact functor between the bounded derived categories of two smooth projective $\K$-schemes $X$ and $X'$. Recall that $\Phi$ is of Fourier--Mukai type if there exists $K \in \Db(X \times X')$ and an isomorphism of functors
\begin{equation} \label{eq_FMfunctor}
\Phi(-) \cong p'_*(K \otimes p^*(-)),    
\end{equation}
where $p \colon X \times X' \to X$, $p' \colon X \times X' \to X'$ are the projections. Let $(\EE, \epsilon)$ and $(\EE', \epsilon')$ be enhancements of $\Db(X)$ and $\Db(X')$, respectively.
\begin{dfn} \label{def_dglift}
A quasi-functor $F \colon \EE \to \EE'$ is a dg lift of $\Phi$ if there is an isomorphism of exact functors $\Phi \cong \epsilon' \circ \rH^0(F) \circ \epsilon^{-1}$.
\end{dfn}

By \cite{LS, To} (see also \cite[Proposition 6.1]{CS_tour}) we have that $\Phi \colon \Db(X) \to \Db(X')$ is a Fourier--Mukai functor if and only if $\Phi$ has a dg lift. 

We suggest the interested readers to consult the excellent survey \cite{CS_tour} for more details and examples on these topics. 

\subsection{Stability conditions on triangulated categories}

Let $\TT$ be a triangulated category. Let us recall the notions of bounded t-structure and stability condition on $\TT$, introduced by Bridgeland in \cite{Bri}.

\begin{dfn}
A t-structure on $\TT$ is a pair of full subcategories $(\DD^{\leq 0}, \DD^{\geq 0})$ satisfying the following conditions:
\begin{enumerate}
    \item $\DD^{\leq 0} \subseteq \DD^{\leq 0}[-1]$ and $\DD^{\geq 0} \supseteq \DD^{\geq 0}[-1]$;
    \item $\Hom_{\TT}(X, Y)=0$ for $X \in \DD^{\leq 0}$, $Y \in \DD^{\geq 0}[-1]$;
    \item For any object $T \in \TT$ there exists an exact triangle $X \to T \to Y \xrightarrow{+}$, where $X \in \DD^{\leq 0}$ and $Y \in \DD^{\geq 0}[-1]$.
\end{enumerate}
The heart of a t-structure $(\DD^{\leq 0}, \DD^{\geq 0})$ is the full subcategory $\AA:=\DD^{\geq 0} \cap \DD^{\leq 0}$. A t-structure is bounded if $$\TT= \bigcup_{a \leq b} \DD^{[a,b]},$$
where $\DD^{[a,b]}:= \DD^{\leq 0}[-b] \cap \DD^{\geq 0}[-a]$.
\end{dfn}

\noindent By \cite{BBD} the heart of a t-structure is an abelian category. Given a t-structure $(\DD^{\leq 0}, \DD^{\geq 0})$, there exist functors 
$$\tau_{\leq a} \colon \TT \to \DD^{\leq 0}[-a], \quad \tau_{\geq a} \colon \TT \to \DD^{\geq 0}[-a]$$
called truncation functors, which are right adjoint and left adjoint to the inclusion functors $\DD^{\leq 0}[-a] \to \TT$ and $\DD^{\geq 0}[-a] \to \TT$, respectively. For every object $T \in \TT$, there exists an exact triangle of the form
$$\tau_{\leq 0}(T) \to T \to \tau_{\geq 1}(T) \xrightarrow{+}$$
(see \cite[IV.4.5.Lemma]{GM}).  

Fix a finite rank lattice with a surjective morphism $\omega \colon K(\TT) \twoheadrightarrow \Lambda$, where $K(\TT)$ denotes the Grothendieck group of $\TT$. 

\begin{dfn} \label{def_stabilitycond}
A stability condition (with respect to $\Lambda$) on $\TT$ is a pair $\sigma=(\AA,Z)$, where $\AA$ is the heart of a bounded t-structure on $\TT$ and $Z \colon \Lambda \to \C$ is a group morphism called central charge, satisfying the following properties:
\begin{enumerate}
    \item For any $0 \neq E \in \AA$ we have $\Im Z\omega(E) \geq 0$, and in the case that $\Im Z\omega(E) = 0$, we have $\Re Z\omega(E) < 0$ (we will write $Z(-)$ instead of $Z\omega(-)$ for simplicity). 
    
    \noindent The slope of $E \in \AA$ is defined as 
    $$\mu_{Z}(E)= 
    \begin{cases}
    -\frac{\Re Z(E)}{\Im Z(E)} & \text{if } \Im Z(E) > 0, \\
    + \infty & \text{otherwise,}
    \end{cases}$$ and a non-zero object $E \in \TT$ is $\sigma$-(semi)stable if $E[k] \in \AA$ for some $k \in \Z$, and for every nonzero proper subobject $F \subset E[k]$ in $\AA$ we have $\mu_{Z}(F) < (\leq) \ \mu_{Z}(E[k]/F)$.
    \item Every object of $E \in \AA$ has a filtration 
    $$0=E_0 \hookrightarrow E_1 \hookrightarrow \dots E_{m-1} \hookrightarrow E_m=E$$
    where $A_i:=E_i/E_{i-1}$ is $\sigma$-semistable and $\mu_Z(A_1) > \dots > \mu_Z(A_m)$. 
    \item There exists a quadratic form $Q$ on $\Lambda \otimes \R$ such that the restriction of $Q$ to the kernel of $Z$ is negative definite and $Q(E) \geq 0$ for all $\sigma$-semistable objects $E$ in $\AA$.
    \end{enumerate}
\end{dfn}
\noindent The objects $A_i$ in Definition \ref{def_stabilitycond} are called Harder--Narasimhan factors of $E$.

Given a stability condition $\sigma=(\AA, Z)$ on $\TT$, we can associate a slicing as follows. Recall that the phase of $E \in \AA$ is 
$$\phi(E)= 
\begin{cases}
\frac{1}{\pi}\text{Arg}(Z(E)) & \text{if } \Im Z(E)>0, \\
1 & \text{otherwise}. 
\end{cases}
$$
If $F=E[k]$ for $E \in \AA$, then $\phi(F)=\phi(E)+k$. We define the collection $\PP_\sigma=\lbrace \PP_\sigma(\phi) \rbrace$ of full additive subcategories $\PP_\sigma(\phi) \subset \TT$ for $\phi \in \R$ such that:
\begin{enumerate}
\item if $\phi \in (0,1]$, the subcategory $\PP_\sigma(\phi)$ is the union of the zero object and all $\sigma$-semistable objects with phase $\phi$;
\item for $\phi+n$ with $\phi \in (0,1]$ and $n \in \Z$, set $\PP_\sigma(\phi+n):=\PP_\sigma(\phi)[n]$.
\end{enumerate}
We write $\PP_\sigma(I)$, where $I \subset \R$ is an interval, to denote the extension-closed subcategory of $\TT$ generated by the subcategories $\PP_\sigma(\phi)$ with $\phi \in I$. Note that $\PP_\sigma((0, 1])= \AA$. 

Note that $\PP_\sigma(\phi)$ has finite lenght for every $\phi \in \R$. In particular, every object $E \in \PP_\sigma(\phi)$ has a (non-unique) finite filtration in $\sigma$-stable objects of the same phase $\phi$, which are called Jordan--H\"older factors.

Now let $X$ be a smooth projective variety defined over the field of complex numbers $\C$. Assume that $\TT$ is a full admissible subcategory of the bounded derived category $\Db(X)$, in other words, the inclusion functor $\TT \to \Db(X)$ is fully faithful and has left and right adjoint. Let $\sigma=(\AA, Z)$ be a stability condition on $\TT$ with respect to the numerical Grothendieck group $\mathrm{K}_{\mathrm{num}}(\TT)$ of $\TT$. For $v \in \mathrm{K}_{\mathrm{num}}(\TT)$, consider the functor 
$$\MM_\sigma(\TT, v) \colon (\text{Sch})^{\text{op}} \to \text{Gpd}$$
from the category of schemes over $\C$ to the category of groupoids, which associates to a scheme $S$ the groupoid $\MM_\sigma(\TT, v)(S)$ of all perfect complexes $E \in \text{D}(X \times S)$ such that, for every $s \in S$, the restriction $E_s$ of $E$ to the fiber $X \times \lbrace s \rbrace$ belongs to $\TT$, is $\sigma$-semistable of phase $\phi$ and $\vv(E_s)=v$. In the examples we will consider in this paper, the functor $\MM_{\sigma}(\TT, v)$ admits a good moduli space $M_\sigma(\TT, v)$, in the sense of \cite{Alper}, which is a proper algebraic space over $\C$. We will denote by $M_\sigma^s(\TT, v)$ the locus of classes of $\sigma$-stable objects in $M_\sigma(\TT, v)$.

Denote by $\text{Stab}(\TT)$ the set of stability conditions on $\TT$ with respect to $\mathrm{K}_{\mathrm{num}}(\TT)$. By Bridgeland Deformation Theorem \cite{Bri} the set $\text{Stab}(\TT)$ (given that it is non-empty) has the structure of complex manifold of dimension equal to the rank of $\mathrm{K}_{\mathrm{num}}(\TT)$. On $\text{Stab}(\TT)$ we have the following right group action. Let $\widetilde{\textrm{GL}}_2^+(\mathbb{R})$ be the universal cover of $\textrm{GL}_2^+(\mathbb{R})$. Given $\widetilde{g}=(g,M) \in \widetilde{\textrm{GL}}_2^+(\mathbb{R})$ with $M \in \textrm{GL}_2^+(\mathbb{R})$ and $g: \mathbb{R} \to \mathbb{R}$ increasing with $g(\phi +1 )= g(\phi) +1$, the action on $\sigma = (\mathcal{P}_\sigma((0,1]), Z) \in \text{Stab}(\TT)$ is given by 
$$\sigma \cdot \widetilde{g} = (\mathcal{P}_\sigma((g(0), g(1)]), M^{-1} \circ Z).$$
In particular, $\sigma$ and $\sigma \cdot \widetilde{g}$ have the same set of semistable objects but with different phases.

\subsection{Semiorthogonal decompositions and Kuznetsov components}

Let $\TT$ be a $\K$-linear triangulated category, where $\K$ is a field. 

\begin{dfn}
A semiorthogonal decomposition for $\TT$, denoted by $\TT = \langle \TT_1, \dots, \TT_m \rangle$, is a sequence of full triangulated subcategories $\TT_1, \dots, \TT_m$ of $\TT$ such that: 
	\begin{enumerate}
		\item $\Hom_{\TT}(E, F) = 0$, for all $E \in \TT_i$, $F \in \TT_j$ and $i>j$;	\item For any $E \in \TT$, there is a sequence of morphisms
		\begin{equation*}  
		0 = E_m \to E_{m-1} \to \cdots \to E_1 \to E_0 = E,
		\end{equation*}
		such that $\mathrm{Cone}(E_i \to E_{i-1}) \in \TT_i$ for $1 \leq i \leq m$.  
	\end{enumerate}
\end{dfn}

\begin{dfn}
An object $E\in\TT$ is exceptional if $\Hom_{\TT}(E,E[k])=0$ for all integers $k \neq0$, and $\Hom_{\TT}(E,E)\cong \K$. An exceptional collection is a collection of objects $E_1,\dots,E_m$ in $\TT$ such that $E_i$ is an exceptional object for all $i$, and $\Hom_{\TT}(E_i,E_j[k])=0$ for all $k$ and $i>j$.
\end{dfn}

Given an exceptional collection $E_1,\dots,E_m$ in $\TT$, we have the semiorthogonal decomposition
$$\TT=\langle \KK, E_1,\dots,E_m \rangle,$$
where $\KK:=\langle E_1,\dots,E_m \rangle^{\perp}= \lbrace F \in \TT: \Hom_{\TT}(E_i, F)=0  \text{ for all } i=1, \dots, m\rbrace$.

We now recall some explicit examples of semiorthogonal decompositions associated to exceptional collections, which define the Kuznetsov components we will consider in the next. In all of them, we assume $X$ is a variety defined over $\C$. \footnote{The following examples can be stated more generally over an algebraically closed field of characteristic $0$ or large enough positive characteristic. However, we need to work over $\C$ to have the results on moduli spaces which we will use in Section \ref{sec_applications}.} 

\begin{ex} \label{ex_cubic4}
Let $X \subset \P^5$ be a cubic fourfold, in other words, a smooth cubic hypersurface in $\P^5$. By \cite{Kuz_cubic} the bounded derived category of $X$ has a semiorthogonal decomposition of the form
$$\Db(X)= \langle \Ku(X), \OO_X, \OO_X(1), \OO_X(2) \rangle,$$
where $\OO_X, \OO_X(1), \OO_X(2)$ are an exceptional collection of line bundles on $X$ and 
$$\Ku(X):=\langle \OO_X, \OO_X(1), \OO_X(2) \rangle^\perp$$
is known as the Kuznetsov component of $X$. The Serre functor of $\Ku(X)$ satisfies $S_{\Ku(X)} \cong [2]$. In particular, $\Ku(X)$ is a noncommutative K3 surface. Stability conditions on $\Ku(X)$ have been constructed in \cite{BLMS} and the associated moduli spaces of stable objects have been studied in \cite{BLMNPS}. 
\end{ex}

\begin{ex} \label{ex_GMeven}
The second example of noncommutative K3 surface is given by the Kuznetsov component of a Gushel--Mukai variety of even dimension. Recall that a GM variety of dimension $2 \leq n \leq 6$ is a smooth intersection of the form
$$X= \text{CG}(2,5) \cap Q \subset \P^{10},$$
where $\text{CG}(2,5)$ denotes the cone over the Grassmannian $\text{G}(2, 5)$ embedded via the Pl\"ucker embedding in a $10$-dimensional projective space $\P^{10}$, and $Q$ is a quadric hypersurface in a projective space $\P^{n+4} \subset \P^{10}$ of dimension $n+4$. By \cite{KP} the bounded derived category of $X$ has the semiorthogonal decomposition 
$$\Db(X)= \langle \Ku(X), \OO_X, \UU_X^\vee, \dots, \OO_X(n-3), \UU^\vee_X(n-3)\rangle.$$
Here $\UU_X^\vee$ denotes the pullback of the dual of the rank-two tautological bundle on the Grassmannian. If $n$ is even, then the Serre functor of $\Ku(X)$ is isomorphic to the homological shift $[2]$. Moreover, stability conditions and their related moduli spaces have been constructed and studied in \cite{PPZ}.
\end{ex}

\begin{ex} \label{ex_GModd}
We can also consider the Kuznetsov component of a GM variety $X$ of odd dimension, namely a GM threefold or fivefold. In this case, we have that $\Ku(X)$ is an Enriques category, since its Serre funtor is the composition of an involutive autoequivalence and the homological shift by $2$ \cite{KP}. Stability conditions on $\Ku(X)$ have been constructed in \cite[Section 6]{BLMS}. 
\end{ex}

\begin{ex} \label{ex_qds}
Another example of Enriques category is given by the Kuznetsov component of a quartic double solid $X$, which is the double cover of $\P^3$ ramified in a quartic surface. By \cite[Corollary 3.5]{Kuz_Fano3} there is a semiorthogonal decomposition of the form
$$\Db(X)= \langle \Ku(X), \OO_X, \OO_X(1) \rangle,$$
where $\Ku(X)$ is the Kuznetov component. Its Serre functor is the composition of an involutive autoequivalence and the homological shift by $2$ by \cite[Corollary 4.6]{Kuz_Calabi}. Again stability conditions on $\Ku(X)$ have been constructed in \cite[Section 6]{BLMS}. 
\end{ex}

\section{Proof of the general results} 

In this section we prove Theorem \ref{thm_derivedcategoryofheart_intro} which is split in Theorem \ref{thm_derivedcategoryofheart} and Theorem \ref{thm_stroglyunique}.

\subsection{Admissible subcategories and hearts}\label{sec_hearts}

Let $\TT$ be the derived category of an abelian category, and let $\KK$ be an admissible subcategory of $\TT$. Suppose that there exists a heart $\AA_\TT$ of a bounded t-structure on $\TT$, and the intersection of $\AA_\TT$ with $\KK$ induces a heart $\AA$ of a bounded t-structure on $\KK$. We denote by $(\DD^{\leq 0}, \DD^{\geq 0})$ the bounded t-structure on $\KK$ whose heart is $\AA$.

The following lemma is a direct consequence of \cite{Beilinson1987}.

\begin{lem} \label{lemma_defF}
There exists a t-exact functor $F:\Db(\AA) \to \KK$, whose restriction to the heart $\AA\subset\Db(\AA)$ is identity to the heart $\AA\subset\KK$.
\end{lem}

\begin{proof}
Since $\TT$ is the derived category of an abelian category, there exists a filtered derived category over $\TT$. This gives an exact functor $F_\TT:\Db(\AA_\TT) \to \TT$, which is t-exact with respect to the standard t-structure on $\Db(\AA_{\TT})$ and that defining the heart $\AA_{\TT}$ in $\TT$, whose restriction to $\AA_\TT$ is the identity \cite[Statement A 6]{Beilinson1987}.

Now $\Db(\AA)$ is a triangulated subcategory of $\Db(\AA_\TT)$ since $\AA \subset \AA_\TT$ is thick and the restriction of $F_\TT$ to $\Db(\AA)\subset \Db(\AA_\TT)$ is identity on $\AA$, so the image of $F_\TT$ is right orthogonal to $^{\perp}\KK\subset \TT$, which implies that the image lies in $\KK$. Hence this restriction gives $F$ with the desired properties.
\end{proof}

\begin{rmk}
The assumption that $\TT$ is the derived category of an abelian category is used to ensure the existence of a filtered derived category over it, which allows to construct the functor using \cite{Beilinson1987}. Alternatively, we can assume that $\TT$ is the homotopy category of a stable $\infty$-category to obtain a similar result.
\end{rmk}

To see when the functor $F$ is an equivalence, we will use the following lemma, which is well-known to the experts, also see in \cite[IV, Exercise 2]{GM}.

\begin{lem}\label{lem:ext_generation}
The functor $F:\Db(\AA) \to \KK$ is an equivalence if and only if for any two objects $A,B$ in $\AA$, and any morphism $f\in \Hom_{\KK}(A, B[n])$ for $n \geq 2$, there exist objects $A_0=A$, $A_1$, $A_2$, ... ,$A_n=B$ in $\AA$, and morphisms $f_i \in \Hom_{\KK}(A_{i-1},A_i[1])$ for $i=1,2,...,n$, such that $f$ is the composition of the $f_i$'s.
\end{lem}
\begin{proof}
We outline the proof for the sake of completeness. Assume that $F$ is an equivalence. Then $F$ induces an isomorphism
$$F_{A,B}^n \colon \Hom_{\Db(\AA)}(A, B[n]) \cong \Hom_\KK(A, B[n])$$
for every pair of objects $A$, $B$ in $\Db(\AA)$, $n \in \Z$. If $A$, $B \in \AA$, then by definition 
$$\Hom_{\Db(\AA)}(A, B[n])= \Ext^n_{\AA}(A, B)$$
and by Yoneda interpretation (see \cite[III.5.Theorem (c)]{GM}), every extension $f' \in \Ext^n_{\AA}(A, B)$ is generated by extensions in $\Ext^1_{\AA}$, in other words, there exist objects $A_0=A$, $A_1$, $A_2$, ... ,$A_n=B$ in $\AA$, and morphisms $f_i' \in \Ext_{\AA}^1(A_{i-1},A_i)$ for $i=1,2,...,n$, such that 
$$f'= f_n'[n-1] \circ \dots f_2'[1] \circ f_1'.$$
Thus for every $f \in \Hom_\KK(A, B[n])$, there exists $f' \in \Ext^n_{\AA}(A, B)$ such that $f=F_{A,B}^n(f')$ and $f'$ is the composition of extensions $f_i'$ as above. Setting $f_i:=F_{A,B}^1(f_i')$, since $F$ is a functor, we have  
$$f=F_{A,B}^n(f')=F_{A,B}^1(f_n')[n-1] \circ \dots F_{A,B}^1(f_2')[1] \circ F_{A,B}^1(f_1')=f_n[n-1] \circ \dots f_2[1] \circ f_1$$
where $f_i \in \Hom_{\KK}(A_{i-1}, A_i[i])$. This proves the first implication.

On the other hand, assume the second condition holds. We first show that $F$ is fully faithful. Indeed, by definition of $\Db(\AA)$, it is enough to show that $F_{A,B}^n$ is an isomorphism for every $A$, $B$ in $\AA$, $n \in \Z$. Note that
$$\Hom_{\KK}(A, B[-n])=0= \Ext^{-n}_{\AA}(A, B) \text{ for }n>0$$
since $\AA$ is the heart of a bounded t-structure, and
$\Hom_{\AA}(A, B) \cong \Hom_{\KK}(A, B)$ since $F$ is the identity on $\AA$. Now note that every $f \in \Hom_{\KK}(A, B[1])$ corresponds to a triangle
$$A \xrightarrow{f} B[1] \to \text{Cone}(f) \xrightarrow{+}.$$
Then $C:=\text{Cone}(f)[-1]$ is in $\AA$, since $A$ and $B$ are. Thus $C$ corresponds to the extension
$$0 \to B \to C \to A \to 0$$
in $\AA$ and defines $f' \in \Ext^1_{\AA}(A, B) = \Hom_{\Db(\AA)}(A, B[1])$. Since $F$ is the identity on $\AA$, it follows that $F_{A,B}^1(f')=f$. If $n \geq 2$, by assumption every $f \in \Hom_{\KK}(A, B[n])$ can be written as a composition
$$f= f_n[n-1] \circ \dots f_2[1] \circ f_1$$
with $f_i \in \Hom_{\KK}(A_{i-1}, A_i[1])$ and $A_0=A$, $A_1$, $A_2$, ... ,$A_n=B$ in $\AA$. Since $F$ induces an identification $\Hom_{\KK}(A_{i-1}, A_i[1]) \cong \Ext^1_{\AA}(A_{i-1}, A_i)$, there exists $f_i' \in \Ext^1_{\AA}(A_{i-1}, A_i)$ such that $f_i=F_{A,B}^1(f_i')$. Setting $f':=f_n'[n-1] \circ \dots f_2'[1] \circ f_1'$, we have 
$$F_{A,B}^n(f')=F_{A,B}^1(f_n')[n-1] \circ \dots F_{A,B}^1(f_2')[1] \circ F_{A,B}^1(f_1')=f_n[n-1] \circ \dots f_2[1] \circ f_1=f.$$
This shows that 
$$F_{A,B}^n \colon \Ext^n_{\AA}(A, B)= \Hom_{\Db(\AA)}(A, B[n]) \to \Hom_{\KK}(A, B[n])$$
is surjective for every $n$. 

Note that $F_{A,B}^n$ is also injective. To show this, we argue by induction on $n$. The case $n \leq 1$ has already been shown. Let $n \geq 2$ and assume $F_{A,B}^m$ is injective for every $m < n$. Let $f' \in \Ext^n_{\AA}(A, B)$ such that $F_{A,B}^n(f')=0$. By Yoneda interpretation $f'$ is of the form
$$f'= f_2'[1] \circ f_1' \colon A \to A_1[1] \to B[n],$$
where $f_1' \in \Ext_{\AA}^1(A,A_1)$, $A_1 \in \AA$ and $f_2' \in \Ext^{n-1}_{\AA}(A_1, B)$. Set $K:=\text{Cone}(f_1')[-1]$. Then we have the short exact sequence 
$$0 \to A_1 \xrightarrow{a} K \to A \to 0$$
in $\AA$ and the exact triangle
$$A \xrightarrow{f_1'} A_1[1] \xrightarrow{a[1]} K[1] \xrightarrow{+}$$
in $\Db(\AA)$. Then we have
$$0=F_{A,B}^n(f')=F_{A,B}^n(f_2'[1] \circ f_1')=F_{A,B}^{n-1}(f_2')[1] \circ F_{A,B}^1(f_1').$$
It follows that $F_{A,B}^{n-1}(f_2')[1]$ lifts to a morphism $g[1] \in \Hom_{\KK}(K[1], B[n])$ such that $$F_{A,B}^{n-1}(f_2')[1]=g[1] \circ F_{A,B}^0(a)[1].$$ 
Since $F_{A,B}^{n-1}$ is surjective, there exists $g' \in \Ext^{n-1}_{\AA}(K, B)$ such that $F_{A,B}^{n-1}(g')=g$. It follows that
$$F_{A,B}^{n-1}(f_2')[1]=F_{A,B}^{n-1}(g')[1] \circ F_{A,B}^0(a)[1]=F_{A,B}^{n-1}(g' \circ a)[1].$$ 
By induction hypothesis, the map $F_{A,B}^{n-1}$ is injective, so we deduce that $f_2'=g' \circ a$. Then we have
$$f'= f_2'[1] \circ f_1'= g'[1] \circ a[1] \circ f_1'=0$$
since $a[1] \circ f_1'=0$. We conclude that $F_{A,B}^n$ is an isomorphism for every $n$, and thus $F$ is fully faithful. 

We now show that $F$ is essentially surjective. We argue as in \cite[Section 3.1.15]{BBD}. 
By definition of bounded t-structure, we have
$$\KK= \bigcup_{a \leq b} \DD^{[a,b]},$$
thus if $K \in \KK$, then $K \in \DD^{[a,b]}$ for some $a \leq b$. We argue by induction on $l=b-a \geq 0$. If $l=0$, then $K=A[a]$ for some $A \in \AA$. Since $F$ is the identity on $\AA$, the object $K$ is in the essential image of $F$. If $l \geq 1$, assume the statement holds for every non negative integer $< l$. Take $a \leq c < b$ and consider the truncation functors $\tau_{\leq c}$ and $\tau_{> c}$. Then we have the triangle
$$\tau_{\leq c}K \to K  \to \tau_{>c}K \xrightarrow{f} \tau_{\leq c}K[1].$$
By the induction hypothesis, there exist $K_1$, $K_2 \in \Db(\AA)$ such that $F(K_1)=\tau_{\leq c}K$ and $F(K_2)=\tau_{>c}K$. Since $F$ is fully faithful there exists $f' \colon K_2 \to K_1[1]$ such that $F(f') \cong f$. Then applying $F$ to the triangle
$$\text{Cone}(f')[-1] \to K_2 \xrightarrow{f'} K_1[1] \xrightarrow{+},$$
we get the commutative diagram
$$
\xymatrix{
F(\text{Cone}(f')[-1]) \ar[r] \ar@{-->}[d]& F(K_2) \ar[r]^{F(f')} \ar[d]^{\cong} & F(K_1)[1] \ar[r]^-+ \ar[d]^{\cong}& \\
K \ar[r] & \tau_{>c}K \ar[r]^{f} & \tau_{\leq c}K[1] \ar[r]^-+ & .
}
$$
By Axiom TR3 of triangulated categories, we have an induced morphism $F(\text{Cone}(f')[-1] \to K$ which is an isomorphism. We conclude that $F$ is essentially surjective, and thus $F$ is an equivalence as we wanted.
\end{proof} 
 
The key observation is that to ensure the condition in Lemma \ref{lem:ext_generation}, it suffices to have a stability condition on $\KK$ with certain special properties. More precisely, suppose that there exists a stability condition $\sigma$ on $\KK$ with heart $\AA$. Denote by $Z$ the central charge of $\sigma$ and by $\mu$ the associated slope. Further assume the following holds:

\begin{customthm}{$(\star)$} \label{star}
$\quad$ 
\begin{enumerate}
    \item[\rm{(a)}] The image of the central charge $Z \colon \mathrm{K}_{\mathrm{num}}(\mathcal K)\rightarrow \C$ is discrete.

    \item[\rm{(b)}] For every non-zero object $E$ in $\mathcal A$ and every real number $s_0$, there exists a $\sigma$-stable object $F$ in $\mathcal A$ satisfying $\mu(F)<s_0$
    and  $\Hom_{\KK}(F,E)\neq 0$.
    
    \item[\rm{(c)}] For any $\sigma$-stable objects $E$ and $F$ in $\mathcal A$, we have $\Hom_{\KK}(E,F[m])=0$ for $m\geq 3$. If $\mu(E)<\mu(F)$ in addition, then we have $\Hom_{\KK}(E,F[2])=0$.
\end{enumerate}
\end{customthm}
We will write $Z=-\mathrm{deg}+i\mathrm{rk}$. Denote by $\delta_0:=\inf\{\mathrm{rk}(E)|E\in \mathcal K, \mathrm{rk}(E)> 0\}$. It is easy to that when $\delta_0\neq0$, the image of $\mathrm{rk}$ in $\R$ is discrete.

We say that $\sigma$ satisfies the Assumption \ref{star} if it has central charge with $\delta_0\neq 0$ satisfying (a-c) as above; or it has central charge with $\delta_0= 0$ and there exists an open neighbourhood $U$ in $\widetilde{\text{GL}}^+(2, \R)$ of the lift of the identity  such that $\sigma\cdot\tilde g$ satisfies  (a-c) for every $\tilde g\in U$. 

Our goal is to show that under the above assumptions, the condition in Lemma \ref{lem:ext_generation} is satisfied, and the functor $F:\Db(\AA) \to \KK$ is an equivalence.


\begin{lem}\label{lem:lowboundspaning}
Let $\sigma$ be a stability condition on $\KK$ satisfying the Assumption \ref{star}. Then for every non-zero object $A$ in $\mathcal A$ and every real number $s$, there exists an object $C$ in $\mathcal A$ satisfying:
\begin{enumerate}
    \item $\mu^+(C)<s$;
    \item there exists a surjective morphism $C\rightarrow A$ in $\mathcal A$.
\end{enumerate}
\end{lem}

\begin{proof}
We first prove the case of $\delta_0\neq 0$. We may assume that image of $Z$ is $\Z+i\Z$. Make induction on $(\mathrm{rk},\mathrm{deg})$ of $A$ with lexicographic order. When $(\mathrm{rk}(A),\mathrm{deg}(A))=(0,1)$, then $A$ is stable. By Assumption (b), there exists a $\sigma$-stable object $C$ in $\mathcal A$ with $\mu(C)<s$ and  $\mathrm{Hom}_{\KK}(C,A)\neq 0$. Since $(\mathrm{rk}(A),\mathrm{deg}(A))=(0,1)$, the object $A$ has no non-trivial quotient object in $\mathcal A$, as it is simple. Therefore, every non-zero morphism from $C$ to $A$ is surjective. 

Now assume the statement holds for all objects with $\mathrm{rk}=0$ and $\mathrm{deg}<m$. Then when $A$ has $(\mathrm{rk},\mathrm{deg})=(0,m)$, by Assumption (b), there is a $\sigma$-stable object $C_1$ with $\mu(C_1)<s$ and $\mathrm{Hom}_{\KK}(C_1,A)\neq 0$. Choose a non-zero morphism, and denote its image in $A$ as $A_1$ and the kernel as $K_1$. In particular, $\mu^+(K_1)<s$ as well.

If $A_1=A$, then there is nothing to prove. Otherwise, we may write $A$ as a short exact sequence $$0\rightarrow A_1\rightarrow A\rightarrow A_2\rightarrow 0$$ in $\mathcal A$, where both $A_1$ and $A_2$ have $\mathrm{rk}=0$ and $\deg<m$. 


By induction, there exists $C_2$ with
$\mu^+(C_2)<\mu^-(K_1)$ and a short  exact sequence $0\rightarrow K_2\rightarrow C_2\rightarrow A_2\rightarrow 0$ in $\mathcal A$. By Assumption (c), $\mathrm{Hom}_{\KK}(C_2,K_1[2])=0$. Therefore, the composition $e:C_2\rightarrow A_2\rightarrow A_1[1]$ 
lifts to $\tilde e: C_2\rightarrow C_1[1]$. In particular, $\tilde e$ commutes the diagram:
\begin{equation}
	\begin{tikzcd}
		C_2 \arrow{d} \arrow{r}{\tilde e}
		&  C_1[1] \arrow{d}\\
		A_2\arrow{r} & A_1[1].
	\end{tikzcd}
\end{equation}
By the octahedral axiom, we have the distinguished triangle
$$K\rightarrow C\rightarrow A\xrightarrow {+},$$ where $C$ is Cone$(C_2[-1]\xrightarrow{\tilde e[-1]}C_1)$ and $K\cong$Cone$(K_2[-1]\xrightarrow{}K_1)$. In particular, both $C$ and $K$ are in $\mathcal A$ with $$\mu^+(C),\mu^+(K)<\max\{\mu^+(C_i),\mu^+(K_i)\}<s.$$

Now we have finish the induction for the case of $\mathrm{rk}=0$. We may assume that the statement holds for all objects with $\mathrm{rk}<r$. When $\mathrm{rk}(A)=r$, we may filterize $A$ as the short exact sequence
$$0\rightarrow A_+\rightarrow A\rightarrow A_-\rightarrow 0$$ in $\mathcal A$, where $\mathrm{rk}(A_+)=0$ and $\mu^+(A_-)<+\infty$. By induction and the same argument as above, if the statement holds for $A_-$, then it will hold for $A$.

Now we may assume $\mu^+(A)<+\infty$. By Assumption (b), there exists a $\sigma$-stable object $C_1$ with $\mu^+(C_1)<s$ and $\Hom_{\KK}(C_1,A)\neq 0$. Choose a non-zero morphism, and denote its image in $A$ as $A_1$ and the quotient as $A_2$. Since $\mu^+(A)<+\infty$, we must have $\mathrm{rk}(A_1)>0$. Hence,  $\mathrm{rk}(A_2)<\mathrm{rk}(A)$. By induction, the statement holds for $A_2$. By the same argument as that in the rank $0$ case, the statement holds for $A$.\\

We then prove the case of $\delta_0=0$. Let $t\in(0,1)$ be a real number sufficiently small such that:
\begin{enumerate}
    \item $-\cot(\pi t)<\min\{s,\mu^-(A)\}$;
    \item the image of $e^{-\pi it}Z$ is infinite on the real axis  and $\sigma_t=(\PP_{\sigma}((t,t+1]),e^{-\pi it}Z)$ satisfies Assumption \ref{star} with $\delta_0\neq 0$.
\end{enumerate}
 Since $-\cot(\pi t)<\mu^-(A)$, the object $A$ is in $\PP_{\sigma}((t,1])\subset\PP_{\sigma}((t,t+1])$. By the statement in the $\delta_0\neq 0$ case, there exists an object $C$ in $\PP_{\sigma}((t,t+1])$ with $\mu_{\sigma_t}^+(C)<-\cot(\cot^{-1}(-s)-t\pi)$ and a surjective morphism $f:C\rightarrow A$ in $\PP_{\sigma}((t,t+1])$.

Note that the object $C$ is in $\PP_{\sigma}((t,1])\subset \AA$ and $\mu_{\sigma}^+(C)=-\cot(\cot^{-1}(-\mu_{\sigma_t}^+(C))+t\pi)<s$. The kernel of $f$ is in $\PP_\sigma((t,s))$. Therefore, the morphism $f$ is surjective in $\AA$ as well. We finish the proof of the statement.
\end{proof}

\begin{cor}
\label{ext2generatedbyext1}
Let $\sigma$ be a stability condition on $\KK$ satisfying the Assumption \ref{star}. Then for every $A, B$ in $\AA$ we have that $\Hom_{\KK}(A, B[2])$ is generated by compositions of extensions between objects in $\AA$.
\end{cor}
\begin{proof}
Let $s=\mu^-(B)>-\infty$, we may pick $C$ as that in Lemma \ref{lem:lowboundspaning}. Let $f: C\to  A$ be the surjective morphism and denote by $K$ the kernel of $f$ in $\mathcal A$. Applying $\Hom_{\KK}(-,B)$ to the short exact sequence $0 \to K \to C \to A \to 0$, we get
$$ \dots \to \Hom_{\KK}(K, B[1]) \to \Hom_{\KK}(A, B[2]) \to \Hom_{\KK}(C, B[2])\to \dots.$$
By the choice of $C$ and Assumption (c), $\Hom_{\KK}(C,B[2])=0$. In particular, the last map $\Hom_{\KK}(K, B[1]) \to \Hom_{\KK}(A, B[2])$ is a surjection and the claim holds.
\end{proof}

As a consequence of the previous computations we get our first main result.

\begin{thm} \label{thm_derivedcategoryofheart}
Let $\TT$ be the derived category of an abelian category. Let $\KK$ be an admissible subcategory of $\TT$ having a stability condition $\sigma=(\AA, Z)$, whose heart $\AA$ is induced from a heart $\AA_{\TT}$ on $\TT$ and satisfying the Assumption \ref{star}. Then the functor $F \colon \Db(\AA) \to \KK$ defined in Lemma \ref{lemma_defF} is an equivalence. 
\end{thm}
\begin{proof}
By Lemma \ref{lem:ext_generation} the functor $F$ defined in Lemma \ref{lemma_defF} is an equivalence if and only if for every $A$, $B$ in $\AA$, $n \geq 2$ we have that $\Hom_{\KK}(A, B[n])$ is generated by degree $1$ extensions of objects in $\AA$. First note that $\Hom_{\KK}(A, B[n])=0$ for every $n \geq 3$. Indeed, up to passing to the stable factors, it is enough to have this vanishing for every pair of $\sigma$-stable objects $A$, $B \in \AA$, which holds by Assumption (c). 

On the other hand, by Corollary \ref{ext2generatedbyext1} we have that $\Hom_{\KK}(A, B[2])$ is generated by compositions of extensions between objects in $\AA$. We conclude that the condition in Lemma \ref{lem:ext_generation} is satisfied and thus $F$ is an equivalence.
\end{proof}

\subsection{Enhancements} \label{sec_enhanc}

Assume $\KK$ satisfies the conditions in Theorem \ref{thm_derivedcategoryofheart}. By \cite{Antieau} the bounded derived category of an abelian category has a unique enhancement (see also \cite{CNS} for the same result without the boundedness condition). Together with Theorem \ref{thm_derivedcategoryofheart} this directly implies that $\KK$ has a unique enhancement. 
 
Using \cite{CS, CNS} we further prove in this section that $\KK$ has a strongly unique enhancement, namely the second part of Theorem \ref{thm_derivedcategoryofheart_intro}.

Let us first recall the notion of almost ample sequence from \cite{CS}.
\begin{dfn}[\cite{CS}, Definition 2.9] \label{def_aaset}
Given an abelian category $\AA$ and a set $I$, a subset $\lbrace C_i \rbrace_{i \in I}$ of objects $C_i \in \AA$ is an \emph{almost ample set} if, for every $A \in \AA$, there exists $i \in I$ satisfying:
\begin{enumerate}
\item[\rm{(i)}] there exist $k \in \mathbb{N}$ and a surjection $C_i^{\oplus k} \twoheadrightarrow A$;
\item[\rm{(ii)}] $\Hom_{\AA}(A, C_i)=0$.
\end{enumerate}
\end{dfn}

Note that by Lemma \ref{lem:lowboundspaning} for any $A \in \AA$ there is $C_A \in \AA$ satisfying conditions (i) and (ii). Indeed, it is enough to choose $s < \mu^-(A)$ and apply Lemma \ref{lem:lowboundspaning} to construct $C_A$ together with a surjection to $A$ and satisfying $\Hom(A, C_A)=0$, as $\mu^+(C_A) < s$. 

Now note that the object $C_A$ only depends on the isomorphism class $[A]$ of $A \in \AA$, namely for $A \cong A' \in \AA$ we have that $C_A$ satisfies (i) and (ii) of Definition \ref{def_aaset} with respect to $A'$. So we change the notation to $C_{[A]}$ and set 
$$I :=\lbrace [A], A \in \AA \rbrace.$$
Since $\AA$ is essentially small, we have that $I$ is a set. Thus the collection
\begin{equation}\label{eq_aaset}
\lbrace C_{[A]} \rbrace_{[A] \in I}    
\end{equation} 
is an almost ample set.

The notion of almost ample set plays a key role in the extension of isomorphisms of functors. The next result is a special case of \cite[Proposition 3.3]{CS}.

\begin{prop}[\cite{CS}, Proposition 3.3] \label{prop_extensionviaampleseq}
Let $\AA$ be an abelian category with finite homological dimension. Assume that $\lbrace C_i \rbrace_{i \in I}$ defines an almost ample set and let $\CC$ be the corresponding full subcategory of $\Db(\AA)$. Let $F$ be an autoequivalence of $\Db(\AA)$ such that there is an isomorphism of functors
$$f \colon F|_{\CC} \xrightarrow{\simeq} \emph{id}_{\CC}.$$
Then there exists an isomorphism of functors $F \xrightarrow{\simeq} \emph{id}_{\Db(\AA)}$ extending $f$.
\end{prop}

Recall that the Verdier quotient $\Ddg(\AA):=\Cdg(\AA) / \mathrm{Ac}_{\mathrm{dg}}^{\mathrm{b}}(\AA)$ is an enhancement of $\Db(\AA)$, and thus by Theorem \ref{thm_derivedcategoryofheart} of $\KK$. Here $\Cdg(\AA)$ denotes the dg category of bounded complexes in $\AA$, and $\mathrm{Ac}_{\mathrm{dg}}^{\mathrm{b}}(\AA) \subset \Cdg(\AA)$ its full dg subcategory of acyclic complexes. In fact, the homotopy category of $\Cdg(\AA)$ is $\rH^0(\Cdg(\AA))=\mathrm{K}^{\mathrm{b}}(\AA)$, where the latter is the homotopy category of complexes. This implies the natural identification $\rH^0(\Ddg(\AA))=\Db(\AA)$ (see for instance \cite[Section 1.2]{CS_tour}).  

\cite{CNS} implies that $\Db(\AA)$ has a unique enhancement, namely if $(\EE, \epsilon)$ is another enhancement of $\Db(\AA)$ then there exists a quasi-functor 
\begin{equation} \label{eq_quasifunctor}
F \colon \Ddg(\AA) \to \EE
\end{equation}
such that $\rH^0(F)$ is an equivalence (thus $F$ is an isomorphism in Hqe), constructed as follows. Let $\rVb(\AA) \subset \rKb(\AA)$ be the full subcategory whose objects have zero differential. Let 
$$Q \colon \rKb(\AA) \to \Db(\AA)$$ 
be the quotient functor and set
$$\rBb(\AA):=Q(\rVb(\AA))$$
which is a full subcategory of $\Db(\AA)$, having the same objects as $\rVb(\AA)$ (but with different morphisms). We use the notation $A^*$ for objects in $\rVb(\AA)$ and thus of $\rBb(\AA)$. We have that the full dg subcategory $\rVbdg(\AA)$ of $\Cdg(\AA)$ consisting of complexes with trivial differential defines an enhancement of $\rVb(\AA)$. In \cite[Section 4]{CNS} the authors construct a dg enhancement $\rBbdg(\AA)$ of $\rBb(\AA)$, whose definition depends on the pair $(\EE, \epsilon)$. Using this construction, they define a functor 
$$g \colon \Cdg(\AA) \to \EE$$
which factors through the quotient $\Ddg(\AA)$, so that $g$ is the composition
$$\Cdg(\AA) \to \Ddg(\AA) \xrightarrow{F} \EE.$$
They finally show that $\rH^0(F) \colon \Db(\AA) \to \rH^0(\EE)$ is an equivalence. 

Our goal is to show the following lemma which implies $\Db(\AA) \cong \KK$ has a strongly unique enhancement. 

\begin{lem} \label{lemma_enhancement}
There is an isomorphism of functors $\epsilon \circ \rH^0(F) \cong \emph{id}_{\Db(\AA)}$.
\end{lem}
\begin{proof}
Set $G':=\epsilon \circ \rH^0(F)$. By Proposition \ref{prop_extensionviaampleseq} it is enough to show there is an isomorphism between the restriction functors $G'|_{\CC} \cong \text{id}_{\CC}$, where $\CC$ is the full subcategory of $\Db(\AA)$ defined by the almost ample set \eqref{eq_aaset}. This is a consequence of the construction of $F$ in \cite{CNS}. Indeed, set $G:=\epsilon \circ \rH^0(g)$. Since by definition $\rH^0(g)= \rH^0(F) \circ Q$, we have $G= G' \circ Q$. By \cite[Lemma 5.1]{CNS} there is an isomorphism of functors
$$\theta \colon G|_{\rVb(\AA)} \xrightarrow{\sim} Q|_{\rVb(\AA)}.$$
As a consequence, for every $A^* \in \rBb(\AA)$, we have that $\theta_{A^*}$ induces an isomorphism $G(A^*) \cong A^*$ between the objects in $\Db(\AA)$. Consider $\alpha \in \Hom_{\Db(\AA)}(A_1^*, A_2^*)$ for $A_i^* \in \rBb(\AA)$. 
By \cite[Proposition 5.2]{CNS} the diagram
$$
\xymatrix{
G(A_1^*) \ar[r]^{G(\alpha)} \ar[d]^{\theta_{A_1^*}} & G(A_2^*) \ar[d]^{\theta_{A_2^*}}\\
A_1^* \ar[r]^\alpha & A_2^*
}
$$
commutes in $\Db(\AA)$. Thus $\theta$ induces a natural transformation $\theta' \colon G'|_{\rBb(\AA)} \cong \text{id}_{\rBb(\AA)}$ which is an isomorphism of functors. Since $\CC$ is a full subcategory of $\rBb(\AA)$, it follows that $\theta'$ induces an isomorphism of functors $f \colon G'|_{\CC} \cong \text{id}_{\CC}$. This implies the statement.
\end{proof}
We are now ready to prove the second part of Theorem \ref{thm_derivedcategoryofheart_intro}.

\begin{thm}\label{thm_stroglyunique}
Let $\TT$ be the derived category of an abelian category. Let $\KK$ be an admissible subcategory of $\TT$ having a stability condition $\sigma=(\AA, Z)$, whose heart $\AA$ is induced from a heart on $\TT$ and satisfying Assumption \ref{star}. Then $\KK$ has a strongly unique enhancement.
\end{thm}

\begin{proof}
Let $(\EE, \epsilon)$ be an enhancement of $\KK$. Consider the quasi-functor $F$ defined in \eqref{eq_quasifunctor}. By Lemma \ref{lemma_enhancement} there is an isomorphism of functors $\epsilon \circ \rH^0(F) \cong \text{id}_{\Db(\AA)}$, giving the statement.
\end{proof}
 
\subsection{Fourier--Mukai functors} \label{sec_FMfunctors}
Let $X_1$ and $X_2$ be smooth projective schemes over a field $\mathbb{K}$. Let $\KK_1 \subset \TT_1:=\Db(X_1)$ and $\KK_2 \subset \TT_2:=\Db(X_2)$ be admissible subcategories which are components of a semiorthogonal decomposition. For $j=1,2$, denote by $i_j^* \colon \TT_j \to \KK_j$ the left adjoint functor of the inclusion $i_j \colon \KK_j \hookrightarrow \TT_j$.

\begin{dfn} \label{def_FMtype}
A fully faithful functor $\Phi \colon \KK_1 \xrightarrow{\sim} \KK_2$ is of Fourier--Mukai type if the composition
$$\TT_1 \xrightarrow{i^*_1} \KK_1 \xrightarrow{\Phi} \KK_2 \xrightarrow{i_2} \TT_2$$
is a Fourier--Mukai functor as in \eqref{eq_FMfunctor}.
\end{dfn} 

Note that if $\Phi$ is a functor of Fourier--Mukai type as in Definition \ref{def_FMtype}, then $\Phi$ is a splitting functor in the sense of \cite[Definition 3.1]{Kuz_hpd}. We now remark the following property which is probably well-known to the experts.  
 
\begin{prop}\label{prop_FMtype}
Assume that $\KK_1$ has a strongly unique enhancement. Then every equivalence $\Phi \colon \KK_1 \xrightarrow{\sim} \KK_2$ is of Fourier--Mukai type.
\end{prop}
\begin{proof}
Let $(\EE_j, \epsilon_j)$ be the natural enhancement of $\TT_j$ for $j=1,2$. Denote by $(\FF_j, \delta_j)$ the enhancement of $\KK_j$ induced from $(\EE_j, \epsilon_j)$. Note that $(\FF_2, \delta_2)$ is an enhancement of $\KK_1$, because of the equivalence $\rH^0(\FF_2) \xrightarrow{\delta_2} \KK_2 \xrightarrow{\Phi^{-1}} \KK_1$. Since $\KK_1$ has a strongly unique enhancement, there exists a quasi-functor $F \colon \FF_1 \xrightarrow{\sim} \FF_2$ such that $\rH^0(F)$ is an equivalence sitting in the following commutative diagram:
$$
\xymatrix{
\rH^0(\FF_1) \ar[r]^{\rH^0(F)}_{\sim} \ar[d]^{\delta_1}_{\sim}& \rH^0(\FF_2) \ar[d]^{\delta_2}_{\sim}\\
\KK_1 \ar[r]^{\Phi}_{\sim} & \KK_2
}.
$$
Note that the inclusion functor $i_2$ has a dg lift $i_2^{\text{dg}}$ making
$$
\xymatrix{
\rH^0(\FF_2) \ar[r]^{\rH^0(i_2^{\text{dg}})} \ar[d]^{\delta_2}_{\sim}& \rH^0(\EE_2) \ar[d]^{\epsilon_2}_{\sim}\\
\KK_2 \ar[r]^{i_2} & \TT_2 
}
$$
to commute. Indeed, $\FF_2$ is by definition the dg subcategory of $\EE_2$ whose objects belong to $\KK_2$ via the equivalence $\epsilon_2$ and is a full admissible subcategory of $\EE_2$. The functor $i_2^{\text{dg}}$ is the natural embedding of $\FF_2$ in $\EE_2$ and $\epsilon_2 \circ \rH^0(i_2^{\text{dg}})$ factors through $\KK_2$ defining $\delta_2$. Analogously, the functor $i_1^*$ has dg lift $i_1^{*\text{dg}}$, defined as the left adjoint of $i_1^{\text{dg}}$. By definition it sits in the commutative diagram
$$
\xymatrix{
\rH^0(\EE_1) \ar[r]^{\rH^0(i_1^{*\text{dg}})} \ar[d]^{\epsilon_1}_{\sim}& \rH^0(\FF_1) \ar[d]^{\delta_1}_{\sim}\\
\TT_1 \ar[r]^{i_1^*} & \KK_1
}.
$$

Putting everything together we have 
$$i_2 \circ \Phi \circ i_1^*=\epsilon_2 \circ \rH^0(i_2^{\text{dg}}) \circ \rH^0(F) \circ \rH^0(i_1^{*\text{dg}}) \circ  (\epsilon_1)^{-1}.$$
Thus $i_2^{\text{dg}} \circ F \circ i_1^{*\text{dg}}$ is a dg lift of $i_2 \circ \Phi \circ i_1^*$. By \cite{To}, \cite[Proposition 6.11]{CS_tour} we conclude that the latter is of FM type. 
\end{proof}

From the previous results we deduce the following characterization.
\begin{cor} \label{cor_FMequivalence}
Let $X_1$ and $X_2$ be smooth projective schemes over a field $\mathbb{K}$. Let $\KK_1$ be an admissible subcategory of $\TT_1:=\Db(X_1)=\Db(\AA_{\TT_1})$ having a stability condition $\sigma=(\AA, Z)$ whose heart $\AA$ is induced from the heart $\AA_{\TT_1}$ on $\TT_1$ and satisfying the Assumption \ref{star}. Let $\KK_2$ be an admissible subcategory of $\Db(X_2)$. Then every equivalence $\KK_1 \xrightarrow{\sim} \KK_2$ is of Fourier--Mukai type.
\end{cor}
\begin{proof}
This is a consequence of Theorem \ref{thm_stroglyunique} and Proposition \ref{prop_FMtype}. 
\end{proof}
 
\section{Geometric applications} \label{sec_applications}
In this section we apply the general results proved in the previous section to interesting geometric situations, listed in Examples \ref{ex_cubic4}, \ref{ex_GMeven}, \ref{ex_GModd}, \ref{ex_qds}, providing the proof of Theorems \ref{thm_application_intro} and \ref{thm_FMtype}. The key point is to show that the Kuznetsov components of the varieties in these examples satisfy Assumption \ref{star}. 

To make a universal argument for most of the cases at once, we make the following assumption on the stability conditions which turns out to be easy to check.
Let $\sigma=(\AA,Z)$ be a stability condition on $\KK$ with heart $\AA$. 
\begin{customthm}{$(\star')$} \label{star2}
$\quad$ 
\begin{enumerate}

    \item[\rm{(a0)}] There exists $\lambda_1$ and $\lambda_2$ in $\mathrm{K}_{\mathrm{num}}(\mathcal A)$ such that the image of $Z$ is a rank $2$ lattice spanned by $Z(\lambda_1)$ and $Z(\lambda_2)$.
    
    \item[\rm{(b1)}] The Euler paring $\chi$ is symmetric on $\mathrm{K}_{\mathrm{num}}(\mathcal A)$ and is negative definite on span$_{\Z}\{\lambda_1,\lambda_2\}$.
    
    \item[\rm{(b2)}] There exists $c\in\Z$ such that for any primitive character $v$ with $\chi(v,v)<c$, we have $M^s_\sigma(\KK, v)\neq \emptyset$. 
    
    \item[\rm{(c)}] Same as that of Assumption ($\star$) (c).
\end{enumerate}
\end{customthm}

Recall that we write $Z=-\mathrm{deg}+i\mathrm{rk}$ and $\delta_0:=\inf\{\mathrm{rk}(E)|E\in \mathcal K, \mathrm{rk}(E)> 0\}$.
\begin{lem}\label{lem:sufficientstable}
Let $\sigma$ be a stability condition on $\KK$ satisfying the Assumption \ref{star2}. Then for every non-zero object $E$ in $\mathcal A$, real number $s$ and $\delta>\delta_0$, there exists a $\sigma$-stable object $F$ in $\mathcal A$ with $\mu(F)<s$ and $\mathrm{rk}(F)<\delta$ such that $\chi(F,E)\neq 0$.
\end{lem}

\begin{proof}
In the case of $\delta_0>0$, by Assumption (a0), there exist non-zero $w$ and $v$ in span$_\Z\{\lambda_1,\lambda_2\}$ such that 
$$\mathrm{rk}(w)=\delta_0\text{ and }Z(v)\in\R_{>0}.$$


Note that $\mathrm{K}_{\mathrm{num}}(\AA)$ can be spanned by $w$, $v$  and elements in $\mathrm{Ker}(Z)$. Since $\chi$ is non-degenerate on $\mathrm{K}_{\mathrm{num}}(\AA)$ by definition, if $\chi(v,E)=0$, then there exists $\kappa_E\in\mathrm{Ker}(Z)$ satisfying $\chi(w+\kappa_E,E)\neq 0$. 

Let $n$ be sufficiently negative such that 
\begin{enumerate}
    \item $\chi(nv+w+\kappa_E,nv+w+\kappa_E) <c$ and $nv+w+\kappa_E$ is primitive; 
    \item $\chi(nv+w+\kappa_E,E)\neq 0$;
    \item $\mu(nv+w+\kappa_E)<s$.
\end{enumerate}

The first condition can be satisfied since the left hand side of the inequality is
$$n^2\chi(v,v)+2n\chi(v,w+\kappa_E)+\chi(w+\kappa_E,w+\kappa_E)$$ and tends to $-\infty$ when $n$ tends to $-\infty$ by Assumption (b1). 

The second condition can be satisfied since either $\chi(v,E)\neq 0$ or $\chi(w+\kappa_E,E)\neq 0$ by the choice of $\kappa_E$.

The third condition can be satisfied since $\lim_{n\rightarrow -\infty}\mu(nv+w+\kappa_E)/n = \mathrm{deg}(v)/\delta_0$.

By Assumption (b2) and the first condition, there exists a $\sigma$-stable object $F\in \mathcal A$  with numerical class $nv+w+\kappa_E$. By the choice of $n$, we have $\mu(F)<s$, $\mathrm{rk}(F)=\delta_0<\delta$,  and $\chi(F,E)\neq 0$.\\

In the case of $\delta_0=0$, there exists a sequence of primitive character $w_n=a_n\lambda_1+b_n\lambda_2$ such that $$\mathrm{rk}(w_n)>0,\; \lim_{n\rightarrow +\infty}\mathrm{rk}(w_n)=0, \;\lim_{n\rightarrow +\infty}\mathrm{deg} (w_n)=-\infty, \; \lim_{n\rightarrow +\infty} \frac{b_n}{a_n}=q\not\in\mathbb Q\text{ and }\lim_{n\rightarrow +\infty} |a_n|=+\infty.$$

Note that $\mathrm{K}_{\mathrm{num}}(\AA)$ can be spanned by  $\lambda_1$, $\lambda_2$ and elements in $\mathrm{Ker}(Z)$. Since $\chi$ is non-degenerate on $\mathrm{K}_{\mathrm{num}}(\AA)$ by definition, if $\chi(\lambda_1,E)=\chi(\lambda_2,E)=0$, then there exists $\kappa_E\in\mathrm{Ker} (Z)$ satisfying $\chi(\kappa_E,E)\neq 0$. If $\chi(\lambda_1,E)$ or $\chi(\lambda_2,E)\neq 0$, then we set $\kappa_E=0$.

Let $n$ be sufficiently large such that 
\begin{enumerate}
    \item $\chi(w_n+\kappa_E,w_n+\kappa_E) <c$; 
    \item $\chi(w_n+\kappa_E,E)\neq 0$;
    \item $\mu(w_n+\kappa_E)<s$  and $\mathrm{rk}(w_n+\kappa_E)<\delta$.
\end{enumerate}

Note that the left hand side of the inequality in the first condition  is $\chi(w_n,w_n)+2\chi(w_n,\kappa_E)+\chi(\kappa_E,\kappa_E)$. Divided by $|a_n|^2$, it tends to $\chi(\lambda_1+q\lambda_2,\lambda_1+q\lambda_2)$ when $n$ tends to $\infty$. This value is negative by Assumption (b1). Therefore the first condition can be satisfied. 

Since $q$ is not a rational number, if  $\chi(\lambda_1,E)$ or $\chi(\lambda_2,E)\neq 0$, then $\chi(w_n,E)$ is not constantly zero. If both $\chi(\lambda_1,E)$ and $\chi(\lambda_2,E)=0$, then  $\chi(\kappa_E,E)\neq 0$ by the choice of $\kappa_E$. Therefore the second condition can be satisfied.

The third condition can be satisfied by the choice of $w_n$.

By Assumption (b2) and the first condition, there exists a $\sigma$-stable object $F\in \mathcal A$  with numerical class $w_n+\kappa_E$. By the choice of $n$ and Assumption (b1), we have $\mu(F)<s$, $\mathrm{rk}(F)<\delta$,  and $\chi(E,F)=\chi(F,E)\neq 0$. Hence, Assumption $(\star)$ (b) holds in all cases of $\delta_0$.
\end{proof}

\begin{prop}\label{lem:assumptioneqvui}
Let $\sigma$ be a stability condition satisfying the Assumption $(\star')$, then for every $\tilde g \in\widetilde{\mathrm{GL}}^+(2,\R)$, the stability condition $\sigma\cdot \tilde g$  satisfies the Assumption $(\star)$.
\end{prop}

\begin{proof}
By definition, Assumption $(\star)$ (a) and (c) hold automatically. Note that the conditions in Assumption \ref{star2} are preserved by the $\widetilde{\mathrm{GL}}^+_2(\R)$-action. We only need to check that $\sigma$ satisfies Assumption $(\star)$ (b). 

By taking the Harder--Narasimhan filtration of $E$ with respect to $\sigma$, we only need prove that Assumption $(\star)$ (b) holds for every $\sigma$-semistable object $E$ in $\AA$. By taking a Jordan--H\"older factor that is also a subobject of $E$, we only need prove that Assumption $(\star)$ (b) holds for every $\sigma$-stable object $E$ in $\AA$.  Namely, we are going to prove the statement that for every $\sigma$-stable object $E$ in $\AA$ and every real number $s_0$, there exists a $\sigma$-stable object $F$ in $\AA$ satisfying $\mu(F)<s_0$ and $\Hom_{\KK}(F,E)\neq 0$.

Fix a real number $s_0$, we define the order $\prec_{s_0}$ for complex numbers in $\R_{>0}\cdot e^{i(0,\pi]}$  as follows:
\begin{itemize}
    \item If $c_1\in \R_{>0}\cdot e^{i(0,\cot^{-1}(-s_0))}$ and  $c_2\in \R_{>0}\cdot e^{i[\cot^{-1}(-s_0),\pi]}$, then $c_1\prec_{s_0}c_2$.
    \item If both $c_j=a_j+ib_j\in \R_{>0}\cdot e^{i[\cot^{-1}(-s_0),\pi]}$ for $j=1,2$, then 
$$c_1\prec_{s_0} c_2 \iff b_1\leq b_2 \text{ and } -\frac{a_1}{b_1}< -\frac{a_2}{b_2}.$$
\end{itemize}

Let us go back to the proof of the statement. Note that if $\mu(E)< s_0$, namely $Z(E)\in \R_{>0}\cdot e^{i(0,\cot^{-1}(-s_0))}$,  then we may just let $F=E$ and there is nothing to prove. 

By Assumption (a), for any complex number $c\in \R_{>0}\cdot e^{i[\cot^{-1}(-s_0),\pi]}$, there are only finitely many $Z(v)$'s of numerical characters $v\in\mathrm{K}_{\mathrm{num}}(\mathcal A)$ satisfying $Z(v)\not\in \R_{>0}\cdot e^{i(0,\cot^{-1}(-s_0))}$ and $Z(v)\prec_{s_0} c$. We may make induction on $Z(E)$ with respect to the order $\prec_{s_0}$.

Assume that for every $\sigma$-stable object $E'\in\mathcal A$ satisfying $Z(E')\prec_{s_0} Z(E)$, the statement holds. In other words, there exists a $\sigma$-stable object $F'$ satisfying $\mu(F')<s_0$ and $\Hom_{\mathcal K}(F',E')\neq 0$. 

In the case that $\delta_0>0$, we set $\delta=\frac{3}2\delta_0$ and $s< \min\{s_0,\frac{1}{\delta_0}\left((\mathrm{rk}(E)+\delta_0)s_0-\mathrm{deg}(E)\right)\}$. In the case that $\delta_0=0$, we set $s=s_0$ and sufficiently small $\delta>0$ such that the image of $Z$ does not intersect the area
$$\{a+ib\;|\;\mathrm{rk}(E)\leq b\leq \mathrm{rk}(E)+\delta,s\leq -\frac{a}b<\mu(E)\}.$$ 

Apply Lemma \ref{lem:sufficientstable} for $E$, $s$ and $\delta$, we get a $\sigma$-stable object  $F_0$ in $\AA$ such that $\mu(F_0)<s$, $\mathrm{rk}(F_0)<\delta$ and $\chi(F_0,E)\neq 0$.

If $\chi(F_0,E)>0$, then by Assumption (c), $\Hom_\KK(F_0,E[m])= 0$ for all $m\neq 0,1$. Hence, $\Hom_\KK(F_0,E)\neq 0$ and the statement holds.

Otherwise, by Assumption $(\star')$ (b1),  we have $\chi(E,F_0)=\chi(F_0,E)<0$ and by Assumption (c), $\Hom_\KK(E,F_0[m])= 0$ for all $m\neq 0,1,2$. Hence, we have $\Hom_\KK(E,F_0[1])\neq 0$.  Let $e$ be a non-zero morphism in $\Hom_\KK(E[-1],F_0)$ and $G$ be Cone$(E[-1]\xrightarrow{e}F_0)$.  In particular, $G$ is in $\mathcal A$.

If $G$ is $\sigma$-stable, then we may let $F=G$. Indeed, we always have $\Hom_\KK(G,E)\neq 0$. In the case of $\delta_0>0$, by the choice of $s$ and $\delta$, we have $\mu(G)=\frac{\deg(E)+\deg(F_0)}{\mathrm{rk}(E)+\delta_0}<\frac{\deg(E)+s\delta_0}{\mathrm{rk}(E)+\delta_0}<s_0$. In the case of $\delta_0=0$, by the choice of $s$ and $\delta$,  we must have $\mu(G)<s_0$. The statement holds.

If $G$ is not $\sigma$-stable, then we consider one of the Jordan--H\"older factors $G^+$ of $G$ with maximal slope that is also a subobject of $G$. Note that $0\neq \Hom_\KK(G^+,G)\hookrightarrow \Hom_\KK(G^+,E)$. We may let $f$ be a non-zero morphism in $\Hom_\KK(G^+,E)$. The image of $f$ is a non-zero subobject $E_0$ of $E$ in $\mathcal A$. Let $E^+_0$ be one of the Jordan--H\"older factors of $E_0$ with maximal slope that is a subobject of $E_0$. Then $E^+_0$ is also a subobject of $E$ in $\mathcal A$ as well.

 In the case of $\delta_0>0$, we have $\mathrm{rk}(G^+)\leq \mathrm{rk}(G)-\delta_0=\mathrm{rk}(E)$. Since $e\neq 0$, we have  $\mu(G^+)<\mu(E)$. Hence, the object $E^+_0$ is a proper subject of $E$. In the case of $\delta_0=0$, it follows by the choice of $s$ and $\delta$ that $\mathrm{rk}(G^+)<\mathrm{rk}(E)$. Hence, the object $E^+_0$ is a proper subject of $E$ in this case as well.
 
 If $\mu(E^+_0)<s_0$, then we may let $F=E^+_0$ and the statement holds. Otherwise, we have $Z(E^+_0)\prec_{s_0} Z(E)$. The statement holds by induction.
\end{proof}

\begin{lem} \label{lemma_Assforcubic}
Let $X$ be a cubic fourfold. Then the stability conditions $\sigma=(\AA, Z)$ on $\Ku(X)$ constructed in \cite{BLMS} satisfy Assumption \ref{star2}.
\end{lem}
\begin{proof} Recall that the numerical Grothendieck group $\mathrm{K}_{\mathrm{num}}(\Ku(X))$  contains two classes $\lambda_1$ and $\lambda_2$ spanning an $A_2$-lattice
$$
\begin{pmatrix}
-2 & 1 \\
1 & -2
\end{pmatrix} 
$$
with respect to  the Euler pairing $\chi$ by \cite{AT}.  The image of the central charge is spanned by the image of $\lambda_1$ and $\lambda_2$. As the Serre duality on $\Ku(X)$ is $[2]$, the Euler pairing on $\mathrm{K}_{\mathrm{num}}(\Ku(X))$ is symmetric. Hence Assumption (a0) and (b1) hold.

Assumption (b2) follows by \cite[Theorem 1.6]{BLMNPS}.

Finally Assumption (c) follows by Serre duality: for any $\sigma$-stable objects $E$ and $F$ in $\mathcal A$, we have $$\Hom_{\KK}(E,F[m])\cong \Hom_{\KK}(F,E[2-m])=0$$ for $m\geq 3$, and if $\mu(E)<\mu(F)$, for $m \geq 2$.
\end{proof}

\begin{lem}
Let $X$ be a GM fourfold. Then the stability conditions $\sigma$ on $\Ku(X)$ constructed in \cite{PPZ} satisfy Assumption \ref{star2}.
\end{lem}
\begin{proof}
 The numerical Grothendieck group $\mathrm{K}_{\mathrm{num}}(\Ku(X))$ contains two classes $\lambda_1$ and $\lambda_2$ spanning an $A_1^{\oplus 2}$-lattice 
$$
\begin{pmatrix}
-2 & 0 \\
0 & -2
\end{pmatrix}
$$
with respect to the Euler pairing $\chi$ by \cite[Lemma 2.27]{KP} and \cite{Pert}.  The image of the central charge is spanned by the image of $\lambda_1$ and $\lambda_2$, see \cite[Section 4]{PPZ}. As the Serre duality on $\Ku(X)$ is $[2]$, the Euler pairing on $\mathrm{K}_{\mathrm{num}}(\Ku(X))$ is symmetric. Hence Assumption (a0) and (b1) hold.

Assumption (b2) follows by \cite[Theorem 1.5]{PPZ}.

Finally, Assumption (c) follows from Serre duality as that in Lemma \ref{lemma_Assforcubic}.
\end{proof}

\begin{lem} \label{lemma_AssGM3}
Let $X$ be a GM threefold. Then the stability conditions $\sigma$ on $\Ku(X)$ constructed in \cite{BLMS} satisfy Assumption \ref{star2}.
\end{lem}
\begin{proof}
Recall that $\mathrm{K}_{\mathrm{num}}(\Ku(X))$ has rank $2$ by \cite[Proposition 3.9]{Kuz_Fano3} and a basis is given by
$$\lambda_1:=1 - \frac{1}{5}H^2, \quad \lambda_2:=2-H +\frac{5}{6}P,$$
where $H$ is the class of a hyperplane in $X$. The intersection matrix with respect to $\chi$ is 
$$\langle \lambda_1,\lambda_2 \rangle= 
\begin{pmatrix}
-1 & 0 \\ 
0 & -1
\end{pmatrix}.$$
 Assumption (a0) and (b1) hold as the central charge is not degenerate.

Assumption (b2) follows by \cite{PPZ_enriques}. More precisely, if $X$ is a GM threefold, then for every pair of coprime integers $c$, $d$ and $v=c \lambda_1 +d \lambda_2$, the moduli space $M_\sigma(\Ku(X),v)$ is non empty and projective of dimension $v^2+1$.

Assumption (c) follows from Serre-duality and the fact that the property of being stable with respect to $\sigma$ is preserved by the Serre functor. Indeed, if $E$ is $\sigma$-stable, then its essential image via the Serre functor $S_{\Ku(X)}(E)$ is $\sigma$-stable by \cite[Theorem 1.1]{PR}. Since $S_{\Ku(X)}^2=[4]$, it is easy to check (see for instance \cite[Lemma 5.9]{PY}) that $\phi(E) < \phi(S_{\Ku(X)}(E)) \leq \phi(E)+2$. Then 
$$\Hom_{\KK}(E,F[m])\cong \Hom_{\KK}(F[m],S_{\Ku(X)}(E))=0$$ for $m\geq 3$, and if $\mu(E)<\mu(F)$, for $m \geq 2$.
\end{proof}

For the quartic double solid case, we do not know whether Assumption \ref{star2} (b2) holds in this case. But we can prove Assumption \ref{star} (b) directly.

\begin{lem}
Let $X$ be a quartic double solid. Let $\sigma$ be the stability condition on $\Ku(X)$ constructed in \cite{BLMS}. Then  for every $\tilde g \in\widetilde{\mathrm{GL}}^+(2,\R)$, the stability condition $\sigma\cdot \tilde g$  satisfies the Assumption $(\star)$.
\end{lem}
\begin{proof}
By \cite[Proposition 3.9]{Kuz_Fano3}, we have that
$$\lambda_1:= 1- \frac{1}{2}H^2, \quad \lambda_2:=H - \frac{1}{2}H^2-\frac{2}{3}P$$ 
is a basis for $\mathrm{K}_{\mathrm{num}}(\Ku(X))$ with intersection form 
$$\langle \lambda_1,\lambda_2 \rangle= 
\begin{pmatrix}
-1 & -1 \\ 
-1 & -2
\end{pmatrix}$$
with respect to $\chi$. Assumption \ref{star2} (a0) and (b1) hold.

Assumption (c) can be proved as in Lemma \ref{lemma_AssGM3} using that $\sigma$ is Serre invariant by \cite[Proposition 5.7]{PY}.

By \cite{PPZ_enriques}, for every pair of coprime integers $c$ and $d$, the moduli space $M_{\sigma\cdot \tilde g}(\Ku(X), 2c \lambda_1+d\lambda_2)$ is non-empty. When $Z_{\sigma\cdot\tilde g}(\lambda_2)\not\in \R$, the argument in Lemma \ref{lem:sufficientstable} still holds. The only difference is on the choice of characters $nv+w$ and $w_n$. In the case of $\delta_0>0$, one may choose $n$ also satisfying $nv+w=2c\lambda_1+d\lambda_2$ for some coprime $c$ and $d$. In the case of $\delta_0=0$, one may choose $w_n=2a_n\lambda_1+b_n\lambda_2$ for coprime $a_n$ and $b_n$. Note that the proof in Proposition \ref{lem:assumptioneqvui} only relies on Assumption \ref{star2} (a0) (b1) (c) and the property in Lemma \ref{lem:sufficientstable}, the stability condition $\sigma\cdot\tilde g$ satisfies Assumption \ref{star} when $Z_{\sigma\cdot\tilde g}(\lambda_2)\not\in \R$.

For the remaining case when $Z_{\sigma\cdot\tilde g}(\lambda_2)\in \R$, by taking a $\widetilde{\mathrm{GL}}^+(2,\R)$-action without disturbing the heart, we may assume that the central charge $Z_{\sigma\cdot\tilde g'}$ is of the form $Z_{\sigma\cdot\tilde g'}(a\lambda_1+b\lambda_2)=-b+ai$.  To simplify notations, we will denote the stability condition by $\sigma$ and the central charge by $Z$.

\begin{lem}\label{lem:doublequarticsolidassump}
Adopt notations as above, then  for every $\sigma$-stable object $E$ in $\AA$ and every real number $s_0$, there exists a $\sigma$-stable object $F$ in $\AA$ satisfying $\mu(F)<s_0$ and $\Hom_{\KK}(F,E)\neq 0$.
\end{lem}

\begin{proof}[Proof of Lemma \ref{lem:doublequarticsolidassump}:]

We will make induction with respect to the partial order $\prec_{s_0}$ as that defined in the proof of Proposition \ref{lem:assumptioneqvui}. 

Assume that for every $\sigma$-stable object $E'\in\mathcal A$ satisfying $Z(E')\prec_{s_0} Z(E)$, the statement holds. 

If there exists a $\sigma$-stable object $F_0$ with character $[F_0]=\lambda_1+s\lambda_2$ satisfying $\chi(F_0,E)\neq 0$, $s<\min\{s_0, (\mathrm{rk}(E)+1)s_0-\mathrm{deg}(E)\}$. 
If $\chi(F_0,E)>0$, then by Assumption (c), $\Hom_\KK(F_0,E)\neq 0$ and the statement holds. If  $\chi(E,F_0)=\chi(F_0,E)<0$, then by Assumption (c),  $\Hom_\KK(E,F_0[1])\neq 0$.  One can show the existence of $F$ satisfying $\mu(F)<s_0$ and $\Hom_{\KK}(F,E)\neq 0$ by the same argument as that in Proposition \ref{lem:assumptioneqvui}.

We may now assume that there is no $\sigma$-stable object $F_0$ with character $[F_0]=\lambda_1+s\lambda_2$ satisfying $\chi(F_0,E)\neq 0$ and $s<\min\{s_0, (\mathrm{rk}(E)+1)s_0-\mathrm{deg}(E)\}$. As $\mathrm{K}_{\mathrm{num}}(\Ku(X))$ is spanned by $\lambda_1$ and $\lambda_2$ and $\chi$ is non-degenerate on $\mathrm{K}_{\mathrm{num}}(\Ku(X))$, there exists $t\in\R$ such that $\chi(E,\lambda_1+t'\lambda_2)\neq 0$ for every $t'\leq t$. Let $s'_0:=\min\{s_0,t,(\mathrm{rk}(E)+1)s_0-\mathrm{deg}(E)\}$. Let $k$ be an odd integer such that $k<\min\{2s'_0,(\mathrm{rk}(E)+2)s'_0-\mathrm{deg}(E)\}$.

By \cite{PPZ_enriques}, there exists a $\sigma$-stable object $F_0$ in $\AA$ with character $[F_0]=2\lambda_1+k\lambda_2$. 

If $\chi(F_0,E)>0$, then by Assumption (c), $\Hom_\KK(F_0,E)\neq 0$ and the statement holds.

Otherwise, since $k<2s'_0\leq 2t$, $\chi(F_0,E)\neq 0$. We must have $\chi(E,F_0)=\chi(F_0,E)<0$. By Assumption (c),  $\Hom_\KK(E,F_0[1])\neq 0$.  Let $e$ be a non-zero morphism in $\Hom_\KK(E[-1],F_0)$ and $G$ be Cone$(E[-1]\xrightarrow{e}F_0)$.  In particular, $G$ is in $\mathcal A$.

If $G$ is $\sigma$-stable, then we may let $F=G$. Indeed, we always have $\Hom_\KK(G,E)\neq 0$. Also, we have $\mu(G)=\frac{\deg(E)+\deg(F_0)}{\mathrm{rk}(E)+\mathrm{rk}(F_0)}=\frac{\deg(E)+k}{\mathrm{rk}(E)+2}<s'_0\leq s_0$. The statement holds.

If $G$ is not $\sigma$-stable, then we consider one of the Jordan--H\"older factors $G^+$ of $G$ with maximal slope that is also a subobject of $G$. Suppose $\mathrm{rk}(G^+)=\mathrm{rk}(E)+1$, then $G/G^+$ is with character $\mathrm{rk}(G/G^+)=1$. Therefore, $G/G^+$ is $\sigma$-stable with  slope less than or equal to $\mu(G)$. Let the character of $G/G^+$ be $\lambda_1+k'\lambda_2$, then $$k'=\mu(G/G^+)\leq \mu(G)=\frac{\deg(E)+k}{\mathrm{rk}(E)+2}<s'_0.$$
This leads to the contradiction since we have assumed that there is no such stable object.

Now we may assume $\mathrm{rk}(G^+)\leq \mathrm{rk}(E)$. Note that $0\neq \Hom_\KK(G^+,G)\hookrightarrow \Hom_\KK(G^+,E)$. We may let $f$ be a non-zero morphism in $\Hom_\KK(G^+,E)$. The image of $f$ is a non-zero subobject $E_0$ of $E$ in $\mathcal A$. Since $e\neq 0$, we have  $\mu(G^+)<\mu(E)$. Since $\mathrm{rk}(G^+)\leq \mathrm{rk}(E)$, $E_0$ is a proper subobject of $E$ in $\AA$.

Let $E^+_0$ be one of the Jordan--H\"older factors of $E_0$ with maximal slope that is a subobject of $E_0$. Then $E^+_0$ is also a subobject of $E$ in $\mathcal A$ as well. It is clear that $Z(E^+_0)\prec_{s_0} Z(E)$. By induction, there exists a $\sigma$-stable object $F$ in $\AA$ satisfying $\mu(F)<s_0$ and $\Hom_{\mathcal K}(F,E^+_0)\neq 0$. Therefore, $\Hom_{\mathcal K}(F,E)\neq 0$ and the statement holds by induction.
\end{proof}
We have proved that $\sigma\cdot\tilde g$ satisfies Assumption \ref{star} (b) for every $\tilde g \in\widetilde{\mathrm{GL}}^+(2,\R)$. The statement holds.
\end{proof}

We are now ready to prove Theorem \ref{thm_application_intro} and Theorem \ref{thm_FMtype}.

\begin{thm} \label{thm_application}
Let $\Ku(X)$ be the Kuznetsov component of a cubic fourfold or of a Gushel--Mukai variety or of a quartic double solid. Then there is an equivalence $F \colon \Ku(X) \cong \Db(\AA)$, where $\AA$ is the heart of a stability condition on $\Ku(X)$ and $F$ is defined in Lemma \ref{lemma_defF}. Moreover, we have that $\Ku(X)$ has a strongly unique enhancement.
\end{thm}
\begin{proof}
Note that if $X$ is a GM fivefold (resp.\ sixfold), then its Kuznetsov component $\Ku(X)$ is equivalent to that of a GM threefold (resp.\ fourfold). This is a consequence of the duality conjecture proved in \cite[Theorem 1.6]{KP_cones}, as explained in \cite[Proof of Theorem 4.18]{PPZ}. Thus we reduce to prove the statement in this case. Then it is a consequence of the above lemmas, Theorem \ref{thm_derivedcategoryofheart} and Theorem \ref{thm_stroglyunique}.
\end{proof}

\begin{proof}[Proof of Theorem \ref{thm_FMtype}]
Assume that $X_1$, $X_2$ are cubic fourfolds or GM varieties of even dimension. Let $F \colon \Ku(X_1) \to \Ku(X_2)$ be a fully faithful exact functor. Then $F$ is an equivalence. Indeed, note that the $0$-th Hochschild cohomology of $\Ku(X_i)$ is $\text{HH}^0(\Ku(X_i))=\C$, thus $\Ku(X_i)$ is connected. Since $\Ku(X_i)$ is Calabi--Yau, by \cite[Proposition 5.1]{Kuz_Calabi} it follows that $F$ is an equivalence. The results then follows from Corollary \ref{cor_FMequivalence}.
\end{proof}

\begin{rmk} \label{rmk_cubic3}
Note that Theorem \ref{thm_derivedcategoryofheart} could potentially be applied to the Kuznetsov component of a cubic threefold. The main missing ingredient is the non-emptyness of moduli spaces of stable objects for the constructed stability conditions. This is part of the work in progress \cite{FLZ}.
\end{rmk}



\bibliographystyle{alpha}
\bibliography{references}

C.L.: Mathematics Institute (WMI), University of Warwick, Coventry, CV4 7AL, United Kingdom.\\
\indent E-mail address: \texttt{C.Li.25@warwick.ac.uk}\\
\indent URL: \texttt{https://sites.google.com/site/chunyili0401/} \\

L.P.: Dipartimento di Matematica ``F.\ Enriques'', Universit\`a degli Studi di Milano, Via Cesare Saldini 50, 20133 Milano, Italy. \\
\indent E-mail address: \texttt{laura.pertusi@unimi.it}\\
\indent URL: \texttt{http://www.mat.unimi.it/users/pertusi} \\

X.Z.: Department of Mathematics, University of California, Santa Barbara, 552 University Rd, Santa Barbara, CA 93117. \\
\indent E-mail address: \texttt{xlzhao@ucsb.edu}\\
\indent URL: \texttt{https://sites.google.com/site/xiaoleizhaoswebsite/} 

\end{document}